\documentclass[11 pt]{article}

\usepackage[utf8]{inputenc}
\usepackage{amsmath,indentfirst,qsymbols,graphicx,psfrag,amsfonts,stmaryrd,hyperref,color,enumerate,amsthm,dsfont}
\usepackage{url}
\usepackage[dvipsnames]{xcolor}
\usepackage{cleveref}
\usepackage{xparse}
\usepackage{todonotes}
\usepackage{comment}

\usepackage[english]{babel}
\usepackage{mdframed,ulem}

\def \Rot{{\sf Rot}}

\newcommand*\Bell{\ensuremath{\boldsymbol\ell}}
\usepackage{ascii}

\def \LAA#1#2#3{{\cal T}_{#2}^{\,\bullet,#3}(#1)}

\def \EA#1#2{{\cal T}_{#2}^{\,\,\textbf{\SYN}\,,d-1}(#1)}
\def \A#1#2{{\cal T}_{#2}(#1)}
\def \a#1#2{|\A{#1}{#2}|}
\def \F#1#2{{\cal F}_{#2}^{\,\bullet\,,d-1}(#1)}
\def \Fp#1#2{{\cal F}_{#2}^{\,\bullet\,,d-1,{\sf Ex}}(#1)}

\def \SS#1#2{{\sf Subset}_{#2}(#1)}
\def \Buds#1{{\sf Buds}(#1)}
\def \Words#1{{\sf Words}(#1)}
\def \Alphabet#1{{\sf Alphabet}(#1)}
\def \bt{\textbf{t}}
\def \Be{\textbf{e}}
\def \Bf{\textbf{f}}
\def \Bg{\textbf{g}}
\def \Ba{\textbf{a}}
\def \Cut{{\sf Cut}}

\def \Rotate{{\sf Rotate}}
\def \AddRoot{{\sf AddRoot}}

\def \Excursions{{\sf Excursions}}

\newtheorem{lem}{Lemma}
\newtheorem{defi}[lem]{Definition}

\def \bls{{\tiny $\blacksquare$ \normalsize }}

\DeclareMathOperator{\argmin}{argmin}

\def \1{\textbf{1}}

\def \Z{\mathbb{Z}}

\def \bt{{\bf t}}

\def \bpar#1{\left\{\begin{array}{#1} }
\def \epar { \end{array}\right.}

\def \app#1#2#3#4#5{\begin{array}{rccl} #1:&#2&\longrightarrow&#3\\ &#4&\longmapsto&#5\end{array}}

\def \ba{\begin{align}}
\def \ea{\end{align}}
\def \be{\begin{eqnarray*}}
\def \ee{\end{eqnarray*}}
\def \ben{\begin{eqnarray}}
\def \een{\end{eqnarray}}

\def \beq{\begin{equation}}
\def \eq{\end{equation}}

\def \build#1#2#3{\mathrel{\mathop{\kern 0pt#1}\limits_{#2}^{#3}}}

\def \ba{{\bf a}}

\def \captionn#1{\begin{center}\begin{minipage}{16.5cm}\sf\caption{\small \textsf{#1}}\end{minipage}\end{center}}

\def \eref#1{(\ref{#1})}
\def \floor#1{\lfloor#1\rfloor}
\def \P{{\mathbb{P}}}

\def \l{\left}

\def \r{\right}
\def \sous#1#2{\mathrel{\mathop{\kern 0pt#1}\limits_{#2}}}
\def \sur#1#2{\mathrel{\mathop{\kern 0pt#1}\limits^{#2}}}

\newqsymbol{`B}{\mathcal{B}}
\newqsymbol{`E}{\mathbb{E}}
\newqsymbol{`I}{\mathbb{I}}
\newqsymbol{`N}{\mathbb{N}}
\newqsymbol{`O}{\Omega}
\newqsymbol{`P}{\mathbb{P}}
\newqsymbol{`Q}{\mathbb{Q}}
\newqsymbol{`R}{\mathbb{R}}
\newqsymbol{`W}{\mathbb{W}}
\newqsymbol{`Z}{\mathbb{Z}}
\newqsymbol{`a}{\alpha}
\newqsymbol{`e}{\varepsilon}
\newqsymbol{`o}{\omega}
\newqsymbol{`t}{\tau}
\newqsymbol{`w}{{\cal W}}
\def\cro#1{\llbracket#1\rrbracket}

\newcommand{\compact}{ \topsep0pt   \itemsep=0pt   \partopsep=0pt   \parsep=0pt}

\newcounter{c}
\def \bir{\begin{itemize}\compact \setcounter{c}{0}}
\def \eir{\end{itemize}\vspace{-2em}~}

\newcounter{d}
\def \bia{\begin{itemize}\compact \setcounter{d}{0}}
\def \eia{\end{itemize}\vspace{-2em}~}

\newcounter{b}
\def \bi{\begin{itemize}\compact \setcounter{b}{0}}

\def \ei{\end{itemize}\vspace{-2em}~}

\begin{document}

\renewcommand{\baselinestretch}{1.15}
\begin{center}
\LARGE
\textbf{Growing random uniform $d$-ary trees} \medbreak
\large \medbreak
\textbf{ Jean-Fran\c{c}ois Marckert}
\medbreak
\large{CNRS, LaBRI, Universit\'e de Bordeaux}
\normalsize 
\end{center}
\renewcommand{\topfraction}{0.75}

\renewcommand{\textfraction}{0.1}
\renewcommand{\floatpagefraction}{0.75}

\begin{abstract} Let $\A{n}{d}$ be the set of $d$-ary rooted trees with $n$ internal nodes. We give a method to construct a sequence $( \bt_{n},n\geq 0)$ where, for any $n\geq 1$, $ \bt_{n}$ has the uniform distribution in $\A{n}{d}$, and $ \bt_{n}$ is constructed from $ \bt_{n-1}$ by the addition of a new node, and a rearrangement of the structure of $ \bt_{n-1}$. This method is inspired by Rémy's algorithm which does this job in the binary case, but it is different from it. This provides a method for the random generation of a uniform $d$-ary tree in $\A{n}{d}$ with a cost linear in $n$. 
\end{abstract}

\section{Introduction}
\label{sec:intro}

Notation: in the paper, we denote by $\cro{a,b}$ the ordered list of integers belonging to $[a,b]\cap \Z$.\medbreak

Rooted planar trees are often represented as on \Cref{fig:tt}: there is a root, and the children of the nodes are distinguishable.
 \begin{figure}[htbp]
   \centerline{\includegraphics{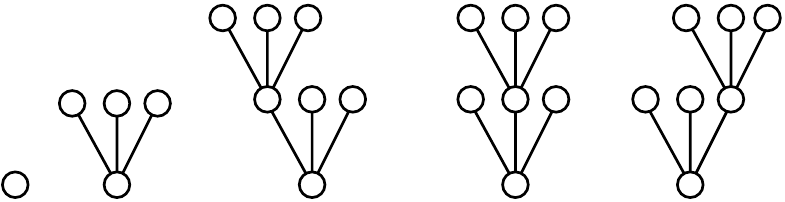}}
   \captionn{\label{fig:tt} The first tree is the single element of $\A{0}{3}$, the next one, the single element of $\A{1}{3}$ and the next tree, the three elements of $\A{2}3$.}
 \end{figure}\par
A $d$-ary tree is a rooted planar tree in which all nodes have either $d$ children or none. Nodes with degree $d$ are called \textit{internal nodes}, the other ones are called {\it leaves}. 
Let  $\A{n}{d}$ be the set of  $d$-ary trees with $n$ internal nodes: $\A{n}{2}$ is the set of standard binary trees, and $\A{n}{3}$ the set of standard ternary trees (see Fig. \ref{fig:tt}). It is well known that, for any $n\geq 0, d\geq 1$, 
\ben\label{eq:dqfrz}
|\A{n}{d}|= \binom{dn+1}{n}/(dn+1);
\een
this can be proved, by decomposing  $d$-ary trees at their root, and by a simple recurrence, or also by the rotation principle (see Sec. \ref{sec:Trp}).
 Each tree $t\in \A{n}{d}$ satisfies
\ben \label{eq:sfhtejku}|t|&=& dn+1,~~ |t^o|=n,~~|\partial t|=(d-1)n+1,~~|E(t)|=dn,\een
where $|t|$ denotes the number of nodes, $t^o$ the set of internal nodes of $t$, $\partial t$ the set of leaves of $t$, and $E(t)$ the set of edges of $t$. 
 
The main aim of this paper is to describe a procedure which produces a random uniform tree $\bt_{n+1}$ in $\A{n+1}{d}$, when one possesses a uniform tree $\bt_{n}$ of $\A{n}{d}$ using some simple operations and the introduction of some (small) additional randomness, of course. The construction we propose, similar in spirit to Rémy's construction \cite{Remy} (binary case), is different from it, even in the case $d=2$ (we discuss the differences in Section \ref{sec:CRB} and provide a third construction in the case $d=2$). By $n$ application of our procedure, it is possible to sample a uniform element of $\A{n}{d}$ starting from a uniform tree in $\A{0}{d}$ (the tree reduced to its root). The details will be given in \Cref{sec:qgegf}.
 
\paragraph{Related works.}
Additionally to Rémy's construction, Marchal \cite{Marchal} proved that Rémy's construction can be used to build an increasing sequence of uniform Dyck path (since uniform binary trees with $n$ internal nodes can be encoded by uniform Dyck path with $2n$ steps), and proved that normalized by $\sqrt{n}$, this sequence of Dyck path converges almost surely in $C[0,1]$ equipped with the topology of uniform convergence. Evans, Grübel and Wakolbinger study the Doob-Martin boundary of Rémy's tree growth chain in \cite{EGW}. 

Bettinelli \cite{JB}, through a bijective approach, gives a method to construct uniform rooted quadrangulations with $n$ faces inductively (from a uniform quadrangulations with $n-1$ faces), and a related method to build uniform forests (with a fixed number of edges).\par

Haas \& Stephenson \cite{H-S} study a model of growing $d$-ary trees, with a construction similar to that of Rémy's, that is, a node with degree $d$ is inserted at each round ``inside a uniformly chosen random edge''; but in the case of $d$-ary tree (with $d\geq 3$), as proved in \cite{H-S}, this method does not produce a sequence of uniform $d$-ary trees; even the order of the height of this model of trees does not fit with uniform $d$-ary tree, since the height order is $n^{1/k}$, when uniform $d$-ary trees have a height of order $\sqrt{n}$ (Aldous \cite{Aldous}).

\paragraph{From a bijection to a growing procedure.}

The construction we propose relies on a new bijection between a set of edge-marked trees with $n$ internal nodes (built over $\A{n}{d}$), and a set of leaf-marked trees with $n+1$ internal nodes (built over $\A{n+1}{d}$). Let us start from the following simple observation: 
 \eref{eq:dqfrz} is equivalent to
\ben\label{eq:deco}
\binom{  \left( d-1 \right) (n+1)+1}{d-1}\times \a{n+1}{d} =d \times \binom{ dn+d-1}{d-1 }\times \a{n}d.
\een
Let us interpret the different elements appearing in this relation. For $S$ a set, and $m\geq 0$, denote by $\SS{S}{m}$ the set of subsets of $S$ with cardinality $m$. 
Of course, $|\SS{S}{m}|=\binom{|S|}{m}$.
\begin{defi} For $n\geq 0$, denote by $\LAA{n}{d}{m}$ the set of pairs $(\bt,\Bell)$ where $\bt\in \A{n}{d}$ and $\Bell\in \SS{\partial \bt}{m}$: in words the set of $d$-ary trees with $n$ internal nodes and  $m$ distinguished leaves. An element of $\LAA{n}{d}{m}$ is called a \underline{$m$-leaf-marked trees} of size $n$.
\end{defi} Since all trees in $\A{n+1}{d}$ have $\left( d-1 \right) (n+1)+1$ leaves, we have
\ben\label{eq:qdqdfeaf}
|\LAA{n+1}{d}{d-1}| &=& \binom{  \left( d-1 \right) (n+1)+1}{d-1}\times \a{n+1}{d},
\een
and this is precisely the left hand side of \eref{eq:deco}.

Introduce a set of cardinality $d-1$, which  can be seen as an ordered list of extra available edges:
\ben{\sf Buds}(d)=\{b_0,\cdots,b_{d-2}\}.
\een
\begin{defi}  For any $n\geq 0$, we denote by $\EA{n}{d}$ the set of pairs $(\bt,\Be)$ where $\bt\in\A{n}{d}$ and $\Be$ is an element of $\SS{E(\bt) \cup \Buds{d}}{d-1}$; this is the set of \underline{edges marked trees} of size $n$.
  Hence, the marks are shared between $\Buds{d}$ and the edge set $E(\bt)$ of $\bt$, and their total number is $d-1$.
\end{defi}
By \eref{eq:sfhtejku}, $|E(\bt) \cup \Buds{d}|=dn+d-1$ for any tree on $\bt\in \A{n}{d}$, so that
\ben\label{eq:sqdrg}
|\EA{n}{d}|= \binom{ dn+d-1}{d-1 } \times |\A{n}d|,
\een
and then the right hand side of \eref{eq:deco} is:
\ben\label{eq:sqdrgdq}
|\EA{n}{d} \times \cro{1,d}|&=& d\times \binom{ dn+d-1}{d-1 } \times |\A{n}d|.
\een
The main part of the rest of the paper is devoted to describe a new bijection between  $\LAA{n+1}{d}{d-1}$ and $\EA{n}{d} \times \cro{1,d}$. It worth probably a moment thought if it is not clear enough: an explicit bijection between these sets allows to construct a procedure to produce a uniform tree $\bt_{n+1}$ in $\A{n+1}{d}$ from a uniform tree $\bt_{n}$ in $\A{n}{d}$ (see \Cref{sec:qgegf}).

To describe our bijection between $\LAA{n+1}{d}{d-1}$ and $\EA{n}{d} \times \cro{1,d}$, we will need two additional families of objects: The set of \underline{leaf marked forests} (see Fig. \ref{fig:manip}).

\begin{defi} We set 
\ben
\F{n}{d}= \bigcup_{n_1,\cdots,n_d\geq 0 \atop{n_1+\cdots+n_d=n}} \bigcup_{
  {m_1,\cdots,m_d\geq 0\atop{m_1+\cdots+m_{d}=d-1}}} \LAA{n_1}{d}{m_1}\times \cdots \times \LAA{n_d}{d}{m_d},
\een that is the \underline{set of leaf-marked forests} $f=((f^{(0)},\ell^{(0)}),\cdots,(f^{(d-1)},\ell^{(d-1)}))$ made of $d$ trees, each of them being a $d$-ary tree, with a total number of $n$ internal nodes, and a total of $d-1$ marked leaves. \par
For such a  leaf-marked forest $f$, define the ``leaf sequence'' $s(f)=(s_i(f),0\leq i \leq d)$, defined by
\ben\label{eq:siF}
s_i(f)&=& (|\ell^{(0)}|-1)+\cdots+(|\ell^{(i-1)}|-1) \textrm{ for }i \in\cro{0,d} 
\een
which is the path starting at 0, and whose increments are the $|\ell^{(j)}|-1$. Since there are $d-1$ marks, we have $s_d(f)=-1$. \par
Denote by $\Fp{n}{d}$ the subset of $\F{n}{d}$ made of \underline{excursion type forests}
\ben
\Fp{n}{d}=\l\{ f \in \F{n}{d}~: s_i(f)\textrm{ for all }i\in\{0,\cdots,d-1\},s_d(f)=-1\r\}.\een
\end{defi}
In fact, to be an excursion type forest is a property concerning the leaf sequence.
\begin{figure}[htbp]
  \centerline{\includegraphics{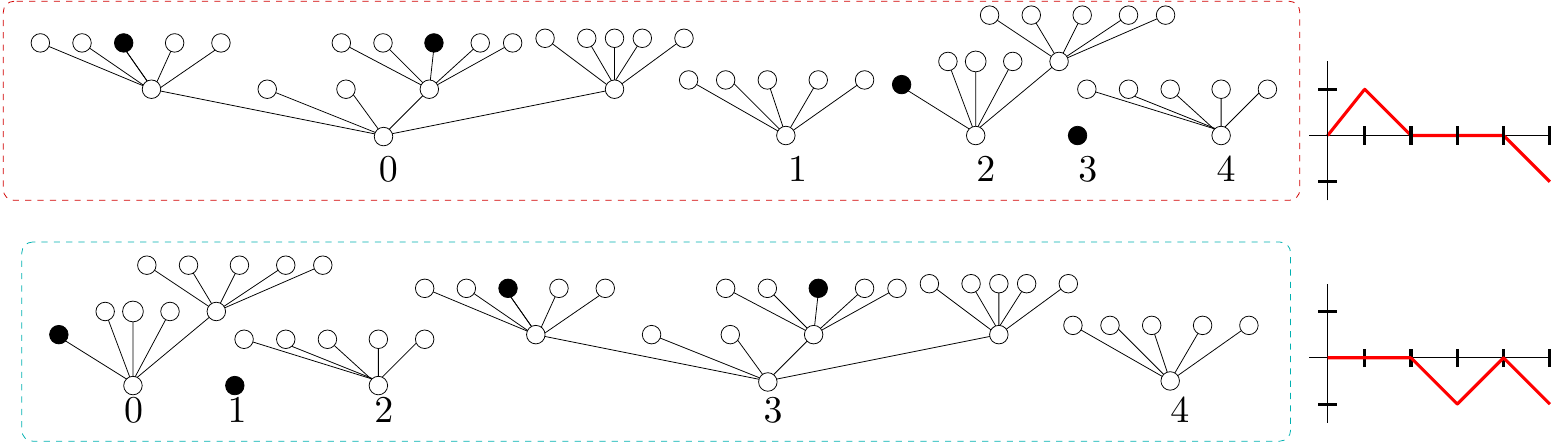}}
  \captionn{\label{fig:manip}
    The first forest belongs to $\F{8}{5}$: it is a 5-ary forest with a total of $8$ internal nodes and 4 marked leaves, and 5 roots. Since $(|\ell^{(0)}|,\cdots,|\ell^{(4)}|)=(2,0,1,1,0)$, the associated leaf sequence has increments $(1,-1,0,0,-1)$ as represented on the path at the right of the forest. The leaf sequence is non negative except at the very last position, so that this forest belongs to $\Fp{8}{5}$. The second forest belongs to  $\F{8}{5}$, but now $(|\ell^{(0)}|,\cdots,|\ell^{(4)}|)=(1,1,0,2,0)$, from what it can be seen that the associated leaf sequence is negative at position 3, so that this forest is not in  $\Fp{8}{5}$.}
  \end{figure}
Since the map which return the list of subtrees rooted at the children of the root, is a bijection from $\LAA{n+1}{d}{d-1}$ onto $\F{n}{d}$, 
\ben\label{eq:dqdfqs}
|\F{n}{d}|= \l|\LAA{n+1}{d}{d-1}\r|.
\een
We also have by the application of the rotation principle (see Section \ref{sec:Trp} for some recall) that
\ben\label{eq:dqftfedq}
|\Fp{n}{d}| = |\F{n}{d}|\, /\, d.
\een
We then have by \eref{eq:deco}, \eref{eq:qdqdfeaf},\eref{eq:sqdrg} and \eref{eq:dqdfqs}
\ben
\l|\EA{n}{d}\times\cro{1,d}\r|= \l|\F{n}{d}\r| = \l|\cro{1,d}\times \Fp{n}{d}\r|=\l|\LAA{n+1}{d}{d-1}\r|.
\een
All these sets having the same cardinality, there exists some bijective correspondences between them, but for the random generation purpose, we want to propose some bijections that preserves as much as possible the forest/tree structures.
We will decompose our bijection between $\EA{n}{d}\times\cro{1,d}$ and $\LAA{n+1}{d}{d-1}$ as the composition of three bijections, 
\be
\begin{array}{ccccc}
  & \Cut & &{\sf Rotate} & \\
  \EA{n}{d}\times\cro{1,d}& \longmapsto & \Fp{n}{d}\times\cro{1,d} &  \longmapsto & \F{n}{d}\\
  \l((\bt_n, \Be), \Ba\r) & \longrightarrow &
\!\!\!\!\l[\l(\Big(\Bg^{(i)},\Bell^{(i)}_{\Bg}\Big),0\leq i \leq d-1\r),\Ba\r]\!\!\!\! & \longrightarrow& \!\!\!\!\l(\Big(\Bf^{(i)},\Bell^{(i)}\Big),0\leq i\leq d-1\r) \\
  & \!\!\!\!\AddRoot\!\!\!\! & & & \\
  & \longmapsto  &\LAA{n+1}{d}{d-1} & & \\
  & \longrightarrow & (\bt_{n+1},\Bell) & &
\end{array}
\ee
{\bf NB}: we should have marked a dependence in $(n,d)$ in these bijections (with an index or an exponent), but renounce to do it, to avoid too heavy notations.\par

In fact, we will see that
\begin{lem}
Each of the map  $\Cut$, $\Rotate$, $\AddRoot$ is a bijection so that 
\ben
{\sf Enlarge}: \EA{n}{d}\times\cro{1,d} \mapsto \LAA{n+1}{d}{d-1}
\een
defined by ${\sf Enlarge}:=\AddRoot \circ \Rotate  \circ  \Cut$
is a bijection, whose inverse is
\ben
{\sf Reduce} := \Cut^{-1} \circ \Rotate^{-1}  \circ \AddRoot^{-1}.
\een
\end{lem}

 Figure \ref{fig:cut} allows to understand the mechanism of the bijection, and it is even possibly sufficient to some readers to understand the complete picture.
\begin{figure}[htbp]
  \centerline{\includegraphics{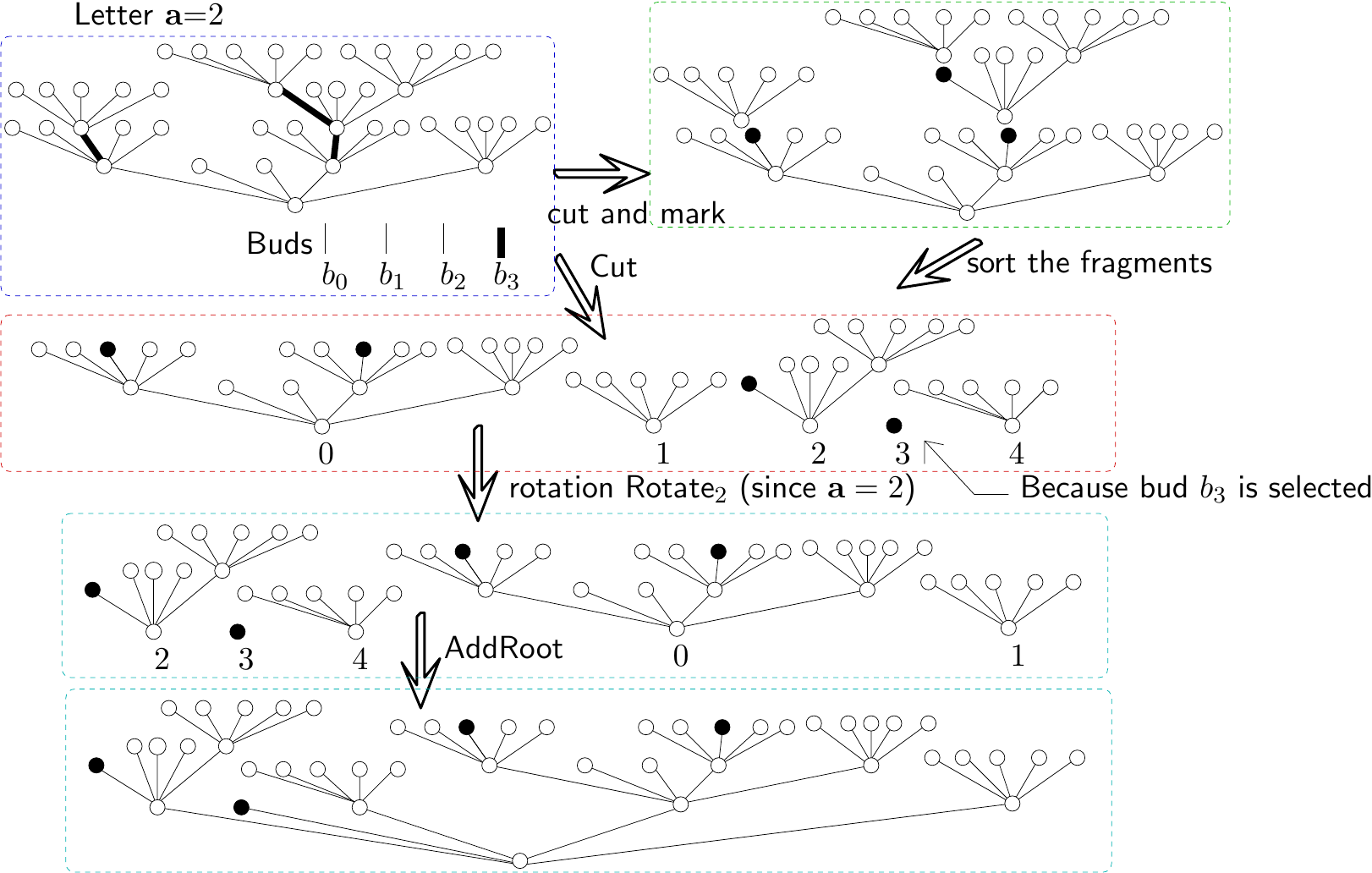}}
  \captionn{\label{fig:cut} Illustration of the $\Cut, \Rotate,\AddRoot$ maps applied to an element $((\bt,\Be),\Ba)$ with $\Ba=2$ of $\EA{8}{5}\times\{1,\cdots,5\}$, in which the buds $b_2$ is marked as well as 3 edges of $\bt$ (darkened).\\
  $\bullet$ \underline{${\sf Cut}$} disconnects the tree at the top extremity of all marked edges, remove the marks on the edges, marks the top extremities of the ancient marked edges, and sort the fragments according to their least vertices for the lexicographical order, in the initial tree. If the buds $(b_{i_j},1\leq j \leq m)$ are selected, then some trees reduced to marked nodes are inserted in the fragment list, so that that these marked nodes represent the trees $\bt_{i_j}$ in the final forest (here $b_3$ is selected so the tree $t_3$ in the list $(t_0,t_1,t_2,t_3,t_4)$ is a simple marked node). This gives an element of  $\Fp{8}{5}$ (since the corresponding walk $s(f)$ satisfies $s(f)=(0,1,0,0,0,-1)$). \\
$\bullet$ \underline{$\Rotate$} is simple enough since it shifts the indices of the forest by $\Ba$ in $\Z/d\Z$, where ${\Ba}$ is the selected letter. It produces an element of $\F{8}{5}$ (here since $\Ba=2$, the image is $(t_2,t_3,t_4,t_0,t_1)$). The fact that it is invertible is not evident since ${\Ba}$ must be recovered from the image only.\\
$\bullet$  \underline{$\AddRoot$} is totally trivial: its name describes its action.\\
Finally, after composition of the three maps, and element of $\LAA{9}{5}{4}$ (9 internal nodes,  4 marked leaves) is produced.}
  \end{figure}

\section{Formal description of the bijections and proofs}

\subsection{Classical encoding of trees}

To define formally the trees, we use the language of set theory (as proposed by Neveu \cite{Neveu1986}). See Fig. \ref{fig:tt2} for an illustration. \begin{figure}[htbp]
   \centerline{\includegraphics{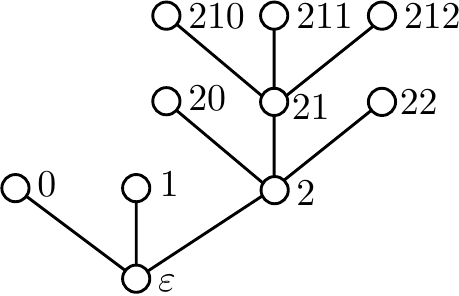}}
   \captionn{\label{fig:tt2}This tree is $\{`e,0,1,2,20,21,22,210,211,212\}$.}
   \end{figure}
Set $\Alphabet{d}=\cro{1,d}$, and the associated set of finite words is
\[\Words{d}= \{`e\}\cup \bigcup_{m\geq 1}\Alphabet{d}^m,\]
where $`e$ is the empty word. The set $\Words{d}$ is the infinite complete $d$-ary tree: each word $w$ stands for a node, and $wj$ stands for the $j$th child of $w$ (see \Cref{fig:tt2}). The vocabulary of genealogy is introduced: if $u$ is a prefix of $v$, then $u$ is called an ancestor of $v$, the largest strict prefix of $v$ is the parent $p(v)$ of $v$; the node $`e$ is the root. \par 
A $d$-ary tree is a subset $T$ of $\Words{d}$ satisfying the following constraints: \\
$\bullet$ $T$ contains $`e$,\\
$\bullet$ $T$ is stable by prefix ( $(u_1,\cdots,u_h)\in T$ implies that $(u_1,\cdots,u_{h-1})\in T$),\\
$\bullet$ $T$ contains the $d$ descendants of each internal node: $(u_1,\cdots,u_h)\in T$ implies that $(u_1,\cdots,u_{h-1},j)\in T$ for all $j \in\Alphabet{d}$.\medbreak

An edge of $T$ is a pair $(u,p(u))$ between a node $u$ of $T$ and its parent $p(u)$, assumed to be directed toward $p(u)$. The lexicographical order makes of each tree a totally ordered set: the prefix of a word is smaller than the word itself, and $`e$ is the smallest element of $\Words{d}$.

\subsection{The map $\Cut$}
Let us describe the map
\ben
\app{\Cut}{\EA{n}{d}\times\cro{1,d}}{\Fp{n}{d}\times\cro{1,d}}{\l((\bt_n, \Be), \Ba\r)}{
  \l[\l(\Big(\Bg^{(i)},\Bell^{(i)}_{\Bg}\Big),0\leq i \leq d-1\r),\Ba\r]}
\een
Observe that the letter $\Ba$ is not modified, and in fact, it is even not used to define $\Cut$. It is present only because we are dealing with some chained bijections, and it is needed to define $\Rotate$. In this section, it can be ignored.\\
\bls Let us concentrate only the definition of $(\bt_n, \Be)\to \l(\Big(\Bg^{(i)},\Bell^{(i)}_{\Bg}\Big),0\leq i \leq d-1\r)$. \par
The set ${\bf e}$ contains $d-1$ elements: $k(1)$ are edges of $\bt_n$ and $k(2)$ are buds $(b_{n_1},\cdots,b_{n_{k(2)}})$ for some $k(1)+k(2)=d-1$. The buds are sorted according to their indices $0\leq n_1<n_2<\cdots <n_{k(2)}\leq d-2$. \par

We decompose the bijection in 2 steps. The two first lines of Fig. \ref{fig:cut} illustrate the construction.\medbreak

\noindent \textbf{Step 1:} for any $ i\in\{1,\cdots,k(2)\}$, the selected bud $b_{n_i}$ is transformed into the tree $\Big(\Bg^{(n_i)},\Bell^{(n_i)}_{\Bg}\Big)$, of the image. This tree is reduced to its root $\Bg^{(n_i)}=\{`e\}$, and it is marked at it, so that $\Bell^{(n_i)}_{\Bg}=\{`e\}$. This produces $k(2)$ trees.\\~\\
\noindent \textbf{Step 2:}  Construct the increasing sequence ${\sf RemainingIndices}$  of the trees $\big(\Bg^{(j)},\Bell^{(j)}_\Bg\big)$ that remains to be built; it is made by the  $d-k(2)=k(1)+1$ elements of $\{0,\cdots,d-1\}\setminus\{n_1,\cdots,n_{k(2}\}$ sorted increasingly.\par
The construction of these marked trees is done by an algorithm.
Some marked edges are progressively transformed into marked leaves, so that the current tree on which we are working, which is an edge marked tree at the beginning, may become temporally a tree marked at some edges and some leaves:~ \textbf{an edge-and-leaf-marked tree} is a 3-tuple $(t,e,\ell)$ where $t\in \cup_n \A{n}{d}$, the marked edges $e\subset (E(t)\cup \Buds{d})$, the marked leaves $\ell \subset \partial t$.
\medbreak
At the beginning, set some variables: 
$(t,e,\ell) :=(\bt_n,\Be,\varnothing),  {\sf Rem}= {\sf RemainingIndices}$. These variables will evolve during the algorithm execution. \par

Initially $|{\sf Rem}|=|e|+1$, and this property will be preserved all along the algorithm execution ($|{\sf Rem}|$ and $|e|$ decrease simultaneously). \\

\centerline{--------------------------}~\\
\noindent {\bf Repeat until the complete disappearance of the elements of $e$}\\
\indent\bls Take $i^\star:= \max\{ {\sf Rem}\}$ the largest remaining index to treat.\\
\indent\bls If $e$ is empty, then there are no marked edge, so that ${\sf Rem}$ contains only $i^\star$. Set $\Big(\Bg^{(i^\star)},\Bell^{(i^\star)}_\Bg\Big)= (t,\ell)$ and the algorithm finishes here.\\
\indent\bls If $e$ is not empty, then find $(u,p(u))$ the largest marked edge of $e$ for the lexicographical order. 
\begin{itemize}
\item[--] First intuitively, the marked tree  $\left(\Bg^{(i^\star)}, \Bell^{(i^\star)}\r)$ is the subtree of $t$ attached at $u$. Formally, we need to express this using our formalism, a tree being a set of words (containing $`e$). Define the set of descendants of $u$ in $t$ is
  \[{\sf Descendants}_t(u):=\{w\in t~:u\textrm{ is a prefix of } w\}\] (including $u$): the marked leaves supported by this set as
  \[\ell(u):=\ell\cap {\sf Descendants}_t(u).\] We have now to remove the prefix $u$ to all the elements of ${\sf Descendants}_t(u)$ to turn it into a tree (in a tree, the root is the empty word $`e$, when in the set ${\sf Descendants}(u)$, the natural root is $u$). 
  Set
  \ben
  \bpar{ccl}\Bg^{(i^\star)}&=&\{w \in \Words{d}, uw \in {\sf  Descendants}_t(u)\},\\
  \Bell^{(i^\star)}&=&\{w \in \Words{d}, uw \in \ell(u)\}.\epar
  \een
\item[--] Now, detach from $t$ the subtree  attached at $u$ (see Fig. \ref{fig:cut}), and update everything as detailed in the four following points:\\
  $\bullet$ set $t:= (t \setminus {\sf  Descendants}(u))\cup \{u\}$ the node $u$ is somehow duplicated, and the descendants of $u$ removed from $t$\\
  $\bullet$ $\ell= (\ell \setminus \ell(u)) \cup \{u\}$ meaning that $u$, which is now a leaf of $t$, is marked and then, added to $\ell$, when the marked leaf of the descendant of $u$ are removed, if any\\
  $\bullet$ Remove $(u,p(u))$ from the set of marked edges $e$ (set $e=e \setminus\{(u,p(u))\}$),\\
  $\bullet$ set ${\sf Rem}:={\sf Rem}\setminus\{i^\star\}$ (the tree $(\Bg^{(i^\star)},\Bell^{(i^\star)})$ is fixed, so that the index $i^\star$ is removed from the set of indices to be treated).
\end{itemize}
\centerline{----------- End of the algorithm -----------}~\\

\begin{lem} $\Cut$ is a bijection.
\end{lem} 
\begin{proof} First, let us see why $\Cut$ is injective: probably the first point to notice is that the trees created by Step 2 can also be reduced to their root, but in this case, this root is not marked. Taking this into account, the fact that the images of different elements in  $\EA{n}{d}\times\cro{1,d}$ are different is simple to see. What is less simple, is the fact that the image of $\Cut$ is indeed included in $\Fp{n}{d}\times\cro{1,d}$, that is $\l((\Bg^{(i)},\Bell^{(i)}),1\leq i \leq d\r)\in \Fp{n}{d}$. In view of \eref{eq:dqftfedq} it would be enough to conclude.\par
  As a matter of fact, it is simple to see that the image  of each edge marked tree from $\EA{n}{d}$ is a forest marked at $d-1$ leaves (each marked bud and marked edge is, at some time, transformed into a marked leaf). Only the fact that this forest has the excursion type needs to be shown. \par
The indices $n_1, \cdots, n_{k(2)}$ of the buds $b_{n_1}, ..., b_{n_{k(2)}}$ becomes the indices of some trees reduced to their roots, marked. Hence, the corresponding increments $|\ell^{(n_i)}|-1$ equal 0. Hence, the fact that the excursion property is satisfied depends basically on the other increments.\\
 Notice that the index $n_{k(2)}<d-1$ since the buds are labeled from 0 to $d-2$: the last tree of the forest $(\Bg^{(d-1)},\Bell^{(d-1)})$ comes from the first detached fragment of $t$, so that it has no marked leaf, and thus, the last increment is $|\ell^{(d-1)}|-1=-1$ (as for all forests of excursion type).

   It remains to prove that  $s_j(f)=\sum_{k=0}^{j-1}|\ell^{(k)}|-1\geq 0$ for $j\leq d-1$. Since, the buds contribution  $|\ell^{(n_i)}|-1$ equal 0, somehow, we can ignore these increments and restrict ourselves to the trees
  $((\Bg^{(m_j)},\Bell^{(m_j)}), 1\leq j \in K)$ where $m_1,\cdots,m_K$ is the list of indices (taken under their initial order) of the trees obtained by the decomposition of $t$ (using Step 2).
  Now, the conclusion is simple: each tree $\Bg^{(m_i)}$ upon creation, creates a leaf in another other component with smaller index (since it is still attached to the current tree $t$). Hence, for any $\ell$,
  \[N_j:=|\ell^{(m_1)}|+\sum_{k=2}^{j-1}(|\ell^{(m_k)}|-1)=s_{m_j}(f)+1\]
  can be interpreted as the number of fragments that were attached to the leaves of the union of the fragments $m_1,\cdots,m_{j-1}$ from what we removed $1$, for the fragments $2$ to $j-1$, so that $N_j$ is the number of fragments that were attached to the $j-1$ first fragments, and which have not been visited yet:
it is clear that $N_j\geq 1$, as long as $j<K-1$ from what the conclusion follows.
\end{proof}

\subsection*{The map $\Cut^{-1}$}

We present the bijection $\Cut^{-1}$ which is simpler:
\ben
  \app{\Cut^{-1}}{\Fp{n}{d}\times\cro{1,d} }{\EA{n}{d}\times\cro{1,d}}{\l[\l(\Big(\Bg^{(i)},\Bell^{(i)}_{\Bg}\Big),0\leq i \leq d-1\r),\Ba\r]}{ \l((\bt_n, \Be), \Ba\r)}
  \een
  Again, the letter $\Ba$ is preserved so that let us focus on the rest.
  Assume that $\l(\Big(\Bg^{(i)},\Bell^{(i)}_{\Bg}\Big),0\leq i \leq d-1\r)$ is given, and let us ``reconstruct'' $(\bt_n, \Be)$.\\
Step 1. Construct the set $S$ of indices $i$ such that  $\Bg^{(i)}=\{`e\}$ and $\Bell^{(i)}_{\Bg}=\{`e\}$, which corresponds to trees reduced to a single node, which is marked. From what is said above, for all such $i\in S$, $i<d-1$. The subset of returned marked buds will be $\{b_i, i \in S\}$: set temporarily $\Be=\{b_i, i \in S\}$, since it is part of the set $\Be$ to be defined.\\
Step 2. For the sequence of ${\sf RemainingIndices} =(m_1,\cdots,m_{d-|S|})$ sorted increasingly.
Set
\[(T,L,E)=(\Bg^{(m_1)},\Bell^{(m_1)},\varnothing)\] a variable of the algorithm, which is an edge and leaf marked tree, which at the beginning, coincides with the right most element of the list of remaining trees to be treated (those that are not reduced to a marked leaf). It will evolve during the algorithm execution; at the beginning it is not marked on any edges.\par

For $k=2$ to $d-|S|$ do the following:\\
$\bullet$ Plug the tree $\Bg^{(m_k)}$ at the first leaf $u$ (for the lex. order) of $T$.\\
$\bullet$ Add the edge $(u,p(u))$ at $E$ (set $E=E\cup\{(u,p(u))\}$). Formally, the tree $T$ obtained after this operation is $T=T+u\Bg^{(m_k)}$, and $L=( L \cup u\Bell^{(m_k)})\setminus\{u\}$ (since the prefix $u$, added to $\Bg^{(m_k)}$ gives a subtree of $T$ isomorphic to $\Bg^{(m_k)}$). \medskip  

At the end set $(\bt_n,\Be)=(T,E)$.  \par

The fact that the function $\Cut^{-1}$ is indeed the inverse map of  $\Cut^{-1}$ should be clear, and is left as an exercise.

\subsection{The map $\Rotate$}
The rotate map is illustrated on Fig. \ref{fig:cut} 
\ben
\app{\Rotate}{\Fp{n}{d}\times\cro{1,d}}{ \F{n}{d}}{\l[\l(\Big(\Bg^{(i)},\Bell^{(i)}_{\Bg}\Big),0\leq i \leq d-1\r),\Ba\r]}{\l(\Big(\Bf^{(i)},\Bell^{(i)}\Big),0\leq i\leq d-1\r)}.
\een
This map is just the rotation of indices (by $\Ba$ in $\Z/d\Z$) defined by
\[\Big(\Bf^{(i)},\Bell^{(i)}\Big):=\l(\Bg^{(i+\textbf{a} \mod d)},\Bell_g^{(i+\textbf{a} \mod d)}\r),~~\textrm{ for all }i \in \Z/d\Z.\]
The fact that this map is invertible is one of the key point of the proof: it is not that obvious because, $\Ba$ needs to be recovered too! The argument is developed in Section \ref{sec:Trp} (second statement of Lemma \ref{lem:rp}). 
Using this Lemma, given an element $F:=\l(\Big(\Bf^{(i)},\Bell^{(i)}\Big),0\leq i\leq d-1\r)\in  \F{n}{d}$, there exists a unique element ${\bf b}\in \cro{1,d}$ such that $\l(\Big(\Bf^{(i+{\bf b} \mod d)},\Bell^{(i+{\bf b}\mod d)}\Big),0\leq i \leq d-1\r)\in \Fp{n}{d}$, and then
\[\Rotate^{-1}\l(\Bf^{(i)},\Bell^{(i)}\r)=\l[\l(\l(\Bf^{(i+{\bf b}\mod d)},\Bell^{(i+{\bf b}\mod d)}\r),0\leq i \leq d-1\r),{\bf b}\r].\]

\subsection{The map $\AddRoot$}
The map $\AddRoot$ is totally trivial and its name suffices to understand its action.
\[\app{\AddRoot}{\F{n}{d}}{\LAA{n+1}{d}{d-1}}{\l(\Big(\Bf^{(i)},\Bell^{(i)}\Big),0\leq i\leq d-1\r)}{(\bt_{n+1},\Bell)}\] 
Formally, we have 
\[\bt_{n+1}=\{`e\}\bigcup_{i=0}^{d-1} i \Bf^{(i)}, ~~~~\Bell=\bigcup_{i=0}^{d-1} i \Bell^{(i)},\]
in words: $\bt_{n+1}$ is the marked tree whose subtrees form the original forest\footnote{again $i\Bf^{(i)}$ is the set by adding $i$ as a prefix to all the nodes of $\Bf^{(i)}$, so that in $\bt_{n+1}$ the subtree rooted at $i$ is isomorphic to $\Bf^{(i)}$}.\\
\bls The reverse map $\app{\AddRoot^{-1}}{\LAA{n+1}{d}{d-1}}{\F{n}{d}}{(\bt_{n+1},\Bell)}{\l(\Big(\Bf^{(0)},\Bell^{(0)}\Big),\cdots,\Big(\Bf^{(d-1)},\Bell^{(d-1)}\Big)\r) }$, amounts to removing the root of $\bt_{n+1}$ while conserving the marked leaves, and to returning the sequence of subtrees of $\bt_{n+1}$ rooted at the children of the root,  according to their initial order.

\subsection{The rotation principle}
\label{sec:Trp}

Set
\ben
{\sf Seq}_m&=&\{s\cro{0,m}~: s_0=0, s_m=-1, s_{j+1}-s_j \geq -1,~ \forall~ 0\leq j \leq m-1\},\\
\Excursions(m) &=& {\sf Seq}_m \cap \{s\cro{0,m}~: s_0\geq 0, \cdots, s_{m-1}\geq 0, s_m=-1\}.
\een  
The first set is sometimes called the set of {\L}ukasiewicz walks, and the second, the set of excursions.
They are set of length $m$ paths with integer values, and increment bounded from below by $-1$, that ends at $-1$; excursions have the additional property to hit $-1$ for the first time at the end.

Denote by $\Delta s_{j-1} = s_j-s_{j-1}$ the $j$th increment of the path $s$. Of course, the increments $(\Delta s_j, 0\leq j \leq m-1)$ characterizes $s\cro{0,m}$, since
\[s_j = \Delta s_0+\cdots+\Delta s_{j-1}.\]
For any $r \in\cro{0,m-1}$, the $r$th rotation is a map on ${\sf Seq}_m$
\[\app{\Rot_r}{{\sf Seq}_m}{{\sf Seq}_m}{s\cro{0,m}}{s'\cro{0,m}}\]
which  is better seen on the increments, which are subjected to a simple rotation around the  $\Z/m\Z$:
\[\Delta s'_i= \Delta s_{i+r \mod m},~~\textrm{ for }i\in\cro{0,m-1}.\]
A rotation class of an element $s\in {\sf Seq}_m$ is defined by
\[{\sf RotationClass}(s)=\{\Rot_0(s),\Rot_1(s),\cdots,\Rot_{m-1}(s)\},\] and two elements $s$ and $s'$ of ${\sf Seq}_m$ are said to be in the same rotation class, if there exists $r \in\cro{0,m-1}$ such that $s=\Rot_r(s')$. Of course $G_m:=\{\Rot_r, 0\leq r \leq m-1\}$
is a group for the composition, isomorphic to $(\Z/m\Z,+)$, so that, it is easily seen that rotation classes are equivalence classes. 

The following result can be found under several forms in the combinatorics literature (rotation principle), and can be found in Otter \cite{Otter}:
\begin{lem}\label{lem:rp} For any $m\geq 1$, each rotation class possesses $m$ different elements, and exactly one of these element belongs to $\Excursions(m)$.\\
  Hence, for any $s\in {\sf Seq}_m$, there exists a unique $a\in \Z/m\Z$ such that $\Rot_{a}(s)\in \Excursions(m)$.
\end{lem}
\begin{proof} Maybe, the simplest proof relies on the action of $\Rot_r$ on the first time $\min \argmin (s)$ a path $s$ hits its minimum for the first time. Observe Fig. \ref{mi}.   Consider the rotation class of a given element $s\in {\sf Seq}(m)$.
  A simple analysis shows that, for $a(s)= \min \argmin (s)$, we have
  \[\Rot_{a(s)}(s) \in \Excursions(m)\]
  so that there is (at least) one excursion in each rotation class, and for any $s\in\Excursions(m)$
  \[a( \Rot_r(s))= m-r,\]
implying that all the elements of the rotation class of an excursion reaches its minimum for the first time at a different place, so that, all of them are different.
  \begin{figure}[htbp]
 \centerline{
\includegraphics[width=14cm, height=6cm]{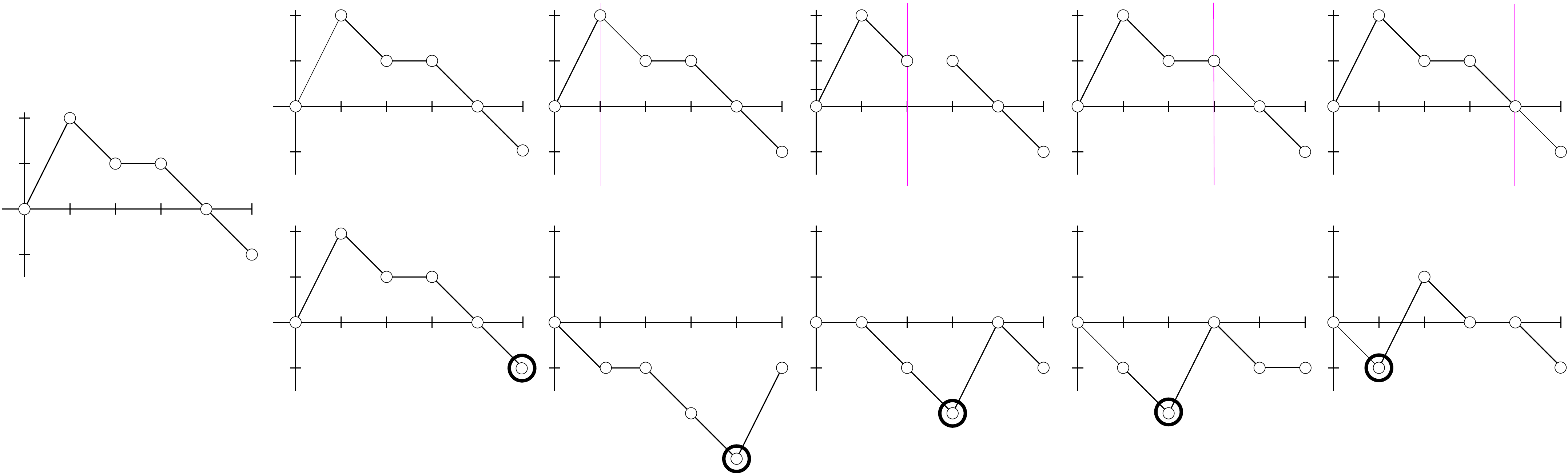}}
\captionn{\label{mi}A path $s$ from $\Excursions(m)$, for $m=5$, and its 5 rotations $\Rot_0(s),\cdots,\Rot_{m-1}(s)$ on the line below: each of them are different (they reach their minimum for the first time at different places).}
\end{figure}

\end{proof}

\section{Random generation of a sequence of growing uniform trees}
\label{sec:qgegf}
We first present the principle of the random generation, and then we will see why the sequence of bijections $({\sf Enlarge}_{k,d},k\geq 0)$ can be used to generate a uniform tree $\bt_{n}$ in $\A{n}{d}$ at a linear cost (for a cost model we will detail).

Assume that $\bt_k$ is uniform in $\A{k}{d}$ (start this recursion at $\bt_0$ the tree reduced to its root $`e$).\\ ~\\
-------- Algo --------\\
choose a uniform subset ${\bf e}$ of $E(\bt_k)\cup \Buds{d}$,\\
2. independently, choose a uniform letter ${\bf a}\in \Alphabet{d}$.\\
3. Compute $(\bt_{k+1},\Bell):={\sf Enlarge}_{k,d}( \bt_k,{\bf e},{\bf a})$.\\
--------------------------\\

When $(\bt_{k+1},\Bell)$ has been computed, the tree $\bt_{k+1}$ is obtained by a simple projection (which amounts to forgetting the marked edges).
\begin{lem} $\bt_{k+1}$ is uniform in $\A{k+1}{d}$\end{lem}
\begin{proof}Each element in the support of $(\bt_k,{\bf e},{\bf a})$ has weight $1/(d \times |\EA{k}{d}|)$, and the number of pairs $(\bt_{k+1},\ell)$ corresponding to a given $\bt_{k+1}=t$ is the same for all trees $t$, that is the number of ways to mark the leaves of $t$, that is $\binom{(d-1)(n+1)+1}{d-1}$: hence by \eref{eq:sqdrg} and \eref{eq:deco}
    \[\P(\bt_{k+1}=t)=\frac{\binom{(d-1)(k+1)+1}{d-1}}{d \times |\EA{k}{d}|}=\frac{1}{\A{k+1}{d}}.\]\end{proof}

\subsection*{About the cost of this generation algorithm}

To define the cost we need to explain a bit how the map ${\sf Enlarge}$ can be programmed, how to proceed to make the number of elementary operations as small as possible, and to fix the cost of the elementary operations. \par
First, ${\sf Enlarge}$ needs to move typically $d-1$ subtrees of $\A{n}{d}$. Hence, some efficient operations are needed to find and move these subtrees. 
If the tree is encoded using some pointers with a link between each node and its parent in the tree, we may assume that the addition of a new node has a constant cost\footnote{it could be also natural to assume that this cost in $O(\log n)$, to take into account the size of the pointers}, and the redirection of some links can be also supposed to have a constant cost, if a table with the list of the nodes is available. The choice of $d$ different edges can be done simply by choosing random nodes, and by identifying each edge with the higher node it contains (to do that, pick some ranks $\floor{(U(dn+d-1))}$ using uniform random variable  $U\sim{\sf Uniform}([0,1])$ till $d-1$ different ranks are obtained); this costs a $O(1)$ number of calls to the random number generator\footnote{an extra cost of $O(\log n)$ to take into account the bit cost of this random generation is also a natural model}.

To get an efficient encoding of the generation algorithm, it is also needed to reach each existing node (those with ranks taken at random, notably) in a fast way. We assume that the cost to reach a given vertex has a constant time \footnote{a cost $O(\log n)$ is also a natural model}.

For this model, the total cost of this algorithm forms to produce a uniform tree in $\A{n}{d}$ from the single tree of size 0, is linear in $n$.

\section{Comparison with Rémy's bijection}
\label{sec:CRB}
This section is devoted to a small discussion of the difference between our algorithm in the case $d=2$ with Rémy's (they are indeed rather different). A third algorithm is presented in section \ref{sec:third}. We skip some details and just talk about the bijection between $|\EA{n}{d}\times\{1,2\}|$ and $\LAA{n+1}{2}{1}$:
trees of size $n$ marked on a edge by $0$ or $1$ and trees of size $n+1$ marked on a leaf.

\subsection{Rémy algorithm}
Rémy's algorithm is illustrated on Fig. \ref{fig:rem}.\\
\begin{figure}[htbp]
   \centerline{\includegraphics{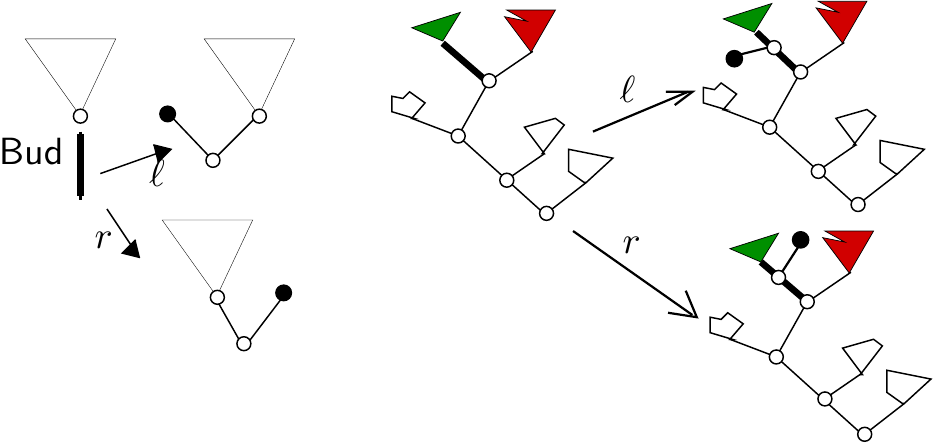}}
   \captionn{\label{fig:rem}Illustration of Rémy algorithm}
 \end{figure}~\vspace{-1mm}
Start with a tree $t$ with $n$ internal nodes, and then, $2n$ edges. There is a marked edge which is in $E(t) \cup {\sf Buds}(2)$ with ${\sf Buds}(2)=\{b_0\}$ which is an available bud. There is a letter $\Ba$ in ${\sf Alphabet}(2)=\{1,2\}$, but in general, in the usual presentation of the bijection, the letter $\Ba$ is rather chosen in $\{r,\ell\}$, ``right'' or ``left''.  \par

The bijection on which Rémy's algorithm is built is
\[\app{R}{\EA{n}{d}\times\{r,\ell\}}{\LAA{n+1}{2}{1}}{(t,e,a)}{(t',f)}\]
where:\\
\bls if the selected edge $e=(u,p(u))$ is in $E(t)$ then do the following: insert a node $w$ ``at the middle'' of the edge $e$, so that $u$ is a child of $w$.\\
. If  $\Ba=r$, then creates a right child of $f$ and marks this leaf (and then $u$ is the left child of $w$),\\
. If  $\Ba=\ell$, then creates a left child of $f$ and marks this leaf (and then $u$ is the right child of $w$),\\
This gives a tree $(t',f)$.\\
\bls if the selected edge is the bud $b_0$, then add a new node $u$ which will become the new root. To build $t'$:\\
-- if $\Ba=r$, add a right edge to $u$ and mark the new leaf $f$ at its extremity. Add a left edge from $u$ to the root of $t$,\\
-- if $\Ba=\ell$, add a left edge to $u$ and mark the new leaf $f$ at its extremity. Add a right edge from $u$ to the root of $t$.
  
\subsection{The new bijection in the binary case}
The new bijection ${\sf Enlarge}$ presented in \Cref{sec:intro} is different from Rémy's bijection because it modifies dramatically the neighborhood of the selected edge $(u,p(u))$, by moving the subtree rooted at $u$ at distance 1 of the root: see Figure \ref{fig:nb}.
 \begin{figure}[htbp]
   \centerline{\includegraphics{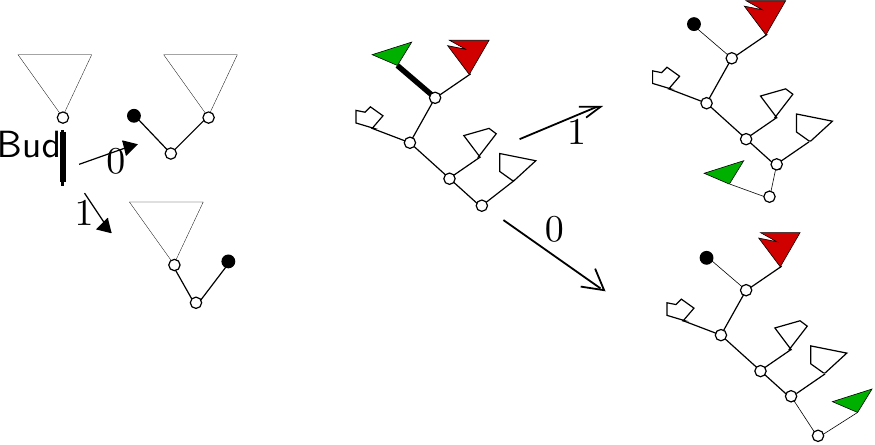}}
   \captionn{\label{fig:nb}Illustration of the new algorithm: the tree at the end of the marked edge becomes a subtree of the root.}
 \end{figure}~\vspace{-1mm}

\subsection{A third bijection in the binary case}
\label{sec:third}
 \begin{figure}[htbp]
   \centerline{\includegraphics{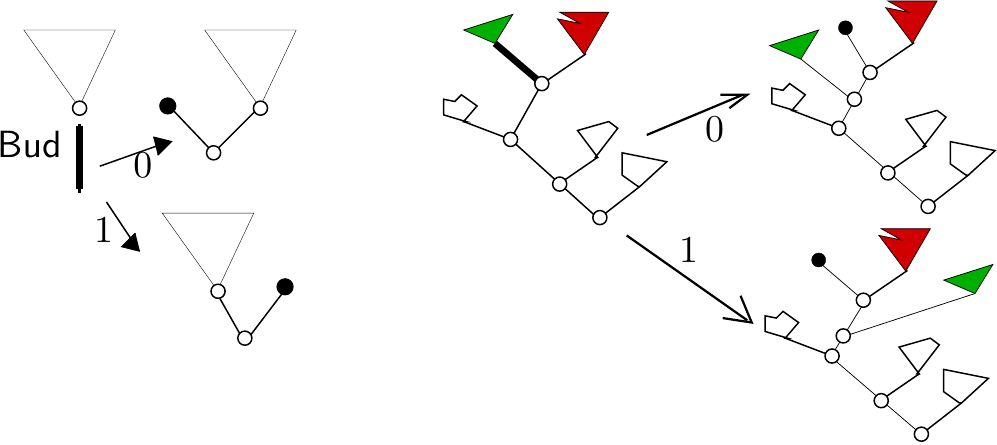}}
   \captionn{\label{fig:var}Illustration of the third algorithm}
 \end{figure}~\vspace{-1mm}

 We present a third bijection different from enlarge, and different from Rémy's, valid in the binary case too.
 It is illustrated on Fig. \ref{fig:var}.\\
 \bls If the marked edge \Be~ is the bud, then do the same thing as in Rémy's bijection,\\
 \bls If the marked edge \Be~ is an edge $(u,p(u))$ of $E(t)$, then do the following:  detach the subtree $T$ grafted at $u$ and mark the node $u$. Insert a new node $v$ in the edge $(p(u),p(p(u)))$, add a node $w$ as new child of $v$ (as a right child if $\Ba=r$, on the left if $\Ba=\ell$), and graft the subtree $T$ at $w$.
 If $p(u)$ is the root of the tree, then $p(p(u))$ is created too, as a new root.

\small
\bibliographystyle{abbrv}

  
\end{document}